\documentclass{amsart}
\usepackage{amsmath, amsthm, amssymb}
\usepackage{graphicx}

\usepackage{ytableau}

\newtheorem{definition}{Definition}
\newtheorem{lemma}{Lemma}
\newtheorem{proposition}[lemma]{Proposition}

\newtheorem{example}{Example}
\newtheorem{thm}{Theorem}
\newtheorem{corollary}{Corollary}
\newtheorem{remark}{Remark}

\newcommand{\bZ}{\mathbb{Z}}
\newcommand{\bC}{\mathbb{C}}

\newcommand{\cm}{\mathfrak{C}_{m,n}}
\newcommand{\am}{\mathbb{A}_{m,n}}
\newcommand{\bm}{\mathbb{B}_{m,n}}
\newcommand{\sn}{\widetilde{S}_n}
\newcommand{\h}{{\bf h}}
\newcommand{\bh}{{\bf \overline{h}}}
\newcommand{\x}{{\bf x}}
\newcommand{\s}{{\bf s}^{(k)}}
\newcommand{\bs}{{\bf s}^{(m,n)}}

\newcommand{\anmz}{A_{(n-m,m)}^0}
\newcommand{\anm}{A_{(n-m,m)}}
\newcommand{\znz}{\bZ/n\bZ}

\DeclareMathOperator{\maxc}{max_c}
\DeclareMathOperator{\maxr}{max_r}

\begin{document}

\title{Positivity of Cylindric skew Schur functions}
\author{Seung Jin Lee}

\address{Department of Mathematical Sciences, Seoul National University, GwanAkRo 1, Gwanak-Gu, Seoul 08826, Korea}      
\email{lsjin@snu.ac.kr}

\maketitle

\begin{abstract}
Cylindric skew Schur functions, a generalization of skew Schur functions, are closely related to the famous problem finding a combinatorial formula for the 3-point Gromov-Witten invariants of Grassmannian. In this paper, we prove cylindric Schur positivity of the cylindric skew Schur functions, conjectured by McNamara. We also show that all coefficients appearing in the expansion are the same as $3$-point Gromov-Witten invariants. We start discussing properties of affine Stanley symmetric functions for general affine permutations and $321$-avoiding affine permutations, and explain how these functions are related to cylindric skew Schur functions. We also  provide an effective algorithm to compute the expansion of the cylindric skew Schur functions in terms of the cylindric Schur functions, and the expansion of affine Stanley symmetric functions in terms of affine Schur functions.
\end{abstract}

\section{introduction}

A cylindric Schur function, a generalization of Schur functions, is a generating function for semistandard Young tableaux on a cylindric shape. It is well-known that Schur functions play an important role on other various subjects such as representation theory and Schubert calculus. For example, the cohomology ring of the Grassmannian can be described in terms of Schur functions and related combinatorics.\\

Postnikov \cite{postnikov.2005} observed that combinatorics on cylindric (skew) Schur functions can describe a quantum cohomology of the Grassmannian $Gr(m,n)$. Postnikov showed that coefficients of a Schur expansion of the cylindric skew Schur polynomials in the first $m$ variables are the same as the multiplicative structure constants of the quantum cohomology of the Grassmannian, called ($3$-point) Gromov-Witten invariants. \\

When there is no restriction on the number of variables, the cylindric skew Schur functions are not often Schur-positive. In other words, most cylindric skew Schur functions cannot be written as a non-negative linear combination of Schur functions. However, McNamara \cite{mcnamara.2006} conjectured that a cylindric skew Schur function is cylindric Schur-positive: it can be written as a non-negative linear combination of cylindric Schur functions. Moreover, these coefficients in the linear combination contain all Gromov-Witten invariants. In this paper, we prove this conjecture and even stronger statement conjectured by McNamara. (See Theorem \ref{thm.main}.)\\

For proving the theorem, we first use the result proved by Lam \cite{lam.2006} that the cylindric skew Schur functions are special cases of the affine Stanley symmetric functions indexed by $321$-avoiding affine permutations. Then we study the affine nilCoxeter algebra to derive interesting identities and symmetries related to the affine Stanley symmetric functions for $321$-avoiding affine permutations.
The cylindric Schur positivity of cylindric skew Schur functions follows from the fact that the affine Stanley symmetric functions can be written as a non-negative linear combination of affine Schur functions and combinatorics of $321$-avoiding affine permutations. Note that finding a combinatorial formula for the coefficients of the expansion of the affine Stanley symmetric functions in terms of affine Schur functions is wide open, and these coefficients contain Littlewood-Richardson coefficients for the flag variety and the Gromov-Witten invariants for the flag variety \cite{peterson.1997} which are famous problems in combinatorics. Although the cylindric Schur positivity of cylindric skew Schur functions is proved in this paper, we do not have a manifestly positive formula for the coefficients in the expansion. Instead, we introduce an effective algorithm to compute these coefficients in Section 5, which also can be used to compute the expansion of affine Stanley symmetric functions in terms of affine Schur functions. 
 \\
 
The structure of the paper is as follows. In Section 2, we introduce the cylindric skew Schur functions and related combinatorics. In Section 3, we review known results about affine Stanley symmetric functions and affine nilCoxeter algebras. In Section 4, we study $321$-avoiding permutations and relate properties of affine Stanley symmetric functions to derive Theorem \ref{thm.main}. In Section 5, we present an effective algorithm to compute the expansion of cylindric skew Schur functions (resp. affine Stanley symmetric functions) in terms of cylindric Schur functions (resp. affine Schur functions). 

\section{Cylindric skew Schur functions}
A \emph{partition} $\lambda=(\lambda_1,\cdots,\lambda_\ell)$ is a weakly decreasing sequence of positive integers. Let $\ell(\lambda)$ be the integer $\ell$, the number of parts in $\lambda$. For positive integers $m<n$, let $P_{mn}$ be a set of partition such that $\lambda_1\leq n-m$ and $\ell(\lambda)\leq m$. The integers $m<n$ are fixed throughout the paper.\\

Let $\cm$ be the set $\bZ^2/ (-n+m, m)\bZ$ and let $\langle i,j\rangle$ be $(i,j)+ (-n+m,m)\bZ$ in $\cm$. $\cm$ inherits a natural partial ordering $<_{\mathfrak{C}}$ from $\bZ^2$, generated by covering relations $\langle i,j\rangle <_{\mathfrak{C}}\langle i+1,j\rangle$ and $\langle i,j\rangle<_{\mathfrak{C}} \langle i,j+1\rangle$. A cylindric diagram $D$ is a finite subset of $\cm$ such that for $a,b \in D$, we have $[a,b]_{\mathfrak{C}} \subset D$. For an integer $r$, $r$-th row of a cylindric diagram $D$ is a set of elements $ \langle r , j \rangle$ in $D$, and $r$-th column of $D$ is a set of elements $ \langle i , r\rangle$ in $D$. Note that $r$-th row of $D$ only depends on $r$ modulo $n-m$ and $r$-th column of $D$ only depends on $r$ modulo $m$. The $r$-th diagonal is a set of elements $\langle i, j \rangle$ satisfying $j-i=r$ modulo $n$.

We say that a bi-infinite integer sequence $\alpha=( \ldots, \alpha_{-1},\alpha_0,\alpha_1,\ldots)$ is $(m,n)$-periodic if $\alpha_i=\alpha_{i+m}-(n-m)$ for any $i \in \bZ$, and increasing if $\alpha_i\leq \alpha_j$ for any $i<j$. For a partition $\lambda \in P_{mn}$ and an integer $r$, let $\lambda[r]$ be the $(m,n)$-periodic sequence defined by $\lambda[r]_{i+r}=\lambda_i+r$ for $i=1,\ldots,m$. All weakly decreasing $(m,n)$-periodic sequences are of the form $\lambda[r]$ for some $r\in \bZ$ and $\lambda\in P_{mn}$. \\

Recall that a subset of a partially ordered set is called an ordered ideal if whenever it contains an element $x$, it also contains all elements less than $x$. There is a one-to-one correspondence between a set $\{ \lambda[r] \mid r \in \bZ, \lambda \in P_{mn}\}$ and the set of an ordered ideal $D_{\lambda[r]} = \{ \langle i,j \rangle \in \cm \mid (i,j) \in \bZ^2, j \leq \lambda[r]_i \}$. For $\lambda,\mu \in P_{mn}$ and $r,s\in \bZ$, we call the pair $(\lambda[r],\mu[s])$ a \emph{cylindric shape} of type $(m,n)$ if $\mu[s]_i\leq \lambda[r]_i$ for all $i$ and denote it by $\lambda[r]/\mu[s]$. The inequality condition is equivalent to the inclusion $D_{\mu[s]} \subset D_{\lambda[r]}$. We denote the set of all cylindric shapes of type $(m,n)$ by $C_{mn}$.
Each cylindric diagram in $\cm$ can be written as the difference of two ordered ideals,
$$D_{\lambda[r]/\mu[s]} = D_{\lambda[r]}/ D_{\mu[s]}= \{\langle i,j \rangle \in \cm \mid (i,j) \in \bZ^2, \mu[s]_i<j \leq \lambda[r]_i \}$$
where $r,s \in \bZ$ and $\lambda,\mu \in  P_{mn}$. We call $\lambda[r]/\mu[s]$ the \emph{shape} of $D_{\lambda[r]/\mu[s]}$. For partitions $\lambda,\mu \in P_{mn}$ and a non-negative integer $d$, we denote $\lambda[d]/\mu[0]$ by $\lambda/d/\mu$.\\

A cylindric semistandard Young tableau of shape $\lambda[r]/\mu[s]$ is the map $T: D_{\lambda[r]/\mu[s]} \rightarrow \mathbb{N}$ satisfying $T(\langle i,j \rangle) \leq T(\langle i+1,j \rangle)$ when $\langle i,j \rangle,\langle i+1,j \rangle \in D_{\lambda[r]/\mu[s]}$ and $T(\langle i,j \rangle) < T(\langle i,j+1\rangle)$ when $\langle i,j \rangle,\langle i,j+1 \rangle \in D_{\lambda[r]/\mu[s]}$ . A sequence $(T^{-1}(1),T^{-1}(2),\ldots )$ is called the \emph{weight} of $T$, denoted by $w(T)$. Let $\x$ be the set of variables $(x_1,x_2,\ldots)$. Now we are ready to define the cylindric skew Schur functions.

\begin{definition} A cylindric skew Schur function $s_{\lambda[r]/\mu[s]}(\x)$ is defined by
$$s_{\lambda[r]/\mu[s]}(\x):= \sum_T \x^{w(T)}$$
where the sum runs over all cylindric semistandard Young tableaux $T$ of shape $\lambda[r]/\mu[s]$, and $\x^{(a_1,a_2,\cdots)}$ is $x_1^{a_1}x_2^{a_2} \cdots$.
\end{definition} 

A cylindric diagram D is called \emph{toric} if each row of $D$ has at most $n-m$ elements and each column of $D$ has at most $m$ elements. If $D_{\lambda/d/\mu}$ is toric, we call the shape $\lambda/d/\mu$ toric as well. 
Let us define a \emph{toric Schur polynomial} as the specialization
$$s_{\lambda/d/\mu}(x_1,\ldots,x_k)=s_{\lambda/d/\mu}(x_1,\ldots,x_k,0,0,\ldots)$$
of the cylindric skew Schur function $s_{\lambda/d/\mu}(\x)$ when $\lambda/d/\mu$ is toric.\\

Recall that $P_{mn}$ is the set of partitions $\lambda$ with $\lambda_1 \leq n-m$ and $\ell(\lambda)\leq m$. Postnikov \cite{postnikov.2005} showed that for $\lambda, \mu \in P_{mn}$, the toric Schur polynomial $s_{\lambda/d/\mu}(x_1,\ldots,x_k)$ is Schur-positive, i.e.,

\begin{align}\label{postnikov}
s_{\lambda/d/\mu}(x_1,\ldots,x_k)= \sum_{\nu \in P_{mn}} C^{\lambda,d}_{\mu,\nu} s_\nu(x_1,\ldots,x_k).
\end{align}

with non-negative integers $C^{\lambda,d}_{\mu,\nu}$. Moreover, $C^{\lambda,d}_{\mu,\nu}$ is the same as the Gromov-Witten invariants for the Grassmannian $Gr(m,n)$.\\

Here we briefly describe the Gromov-Witten invariants $C^{\lambda,d}_{\mu,\nu}$. Let $Gr(m,n)$ be the set of $m$-dimensional subspaces in $\bC^n$. The set $Gr(m,n)$ is a complex projective variety called the \emph{Grassmannian}. There is a cellular decomposition of $Gr(m,n)$ into Schubert cells $\Omega_\lambda^\circ$ where $\lambda$ is a partition in $P_{mn}$. Let $\Omega_\lambda$ be the Zariski closure of $\Omega_\lambda^\circ$, called the Schubert variety.

For a partition $\lambda \in P_{mn}$, let $\lambda^\vee$ be the partition in $P_{mn}$ defined by $\lambda^\vee_i= n-m- \lambda_{m-i+1}$ for $i=1,\ldots,m$. For partitions $\lambda,\mu,\nu \in P_{mn}$, the Gromov-Witten invariant $C^{\lambda,d}_{\mu,\nu}$ is the number of rational curves of degree $d$ passing through generic translates of the Schubert varieties $\Omega_{\lambda^\vee}, \Omega_\mu$ and $\Omega_\nu$ in the Grassmannian $Gr(m,n)$, when the number is finite. It implies that $C^{\lambda,d}_{\mu,\nu}$ is zero unless $|\lambda|+nd=|\mu|+|\nu|$. For $d=0$, the constant $C^{\lambda,d}_{\mu,\nu}$ is the same as the Littlewood-Richardson coefficients. \\

In general, cylindric skew Schur functions are not Schur-positive. However, McNamara \cite{mcnamara.2006} conjectured the following statement which is our main theorem.

\begin{thm} \label{thm.main}
For a cylindric shape $\lambda/d/\mu$ in $\cm$, let $s_{\lambda/d/\mu}$ be the cylindric skew Schur functions. Then 

\begin{enumerate}
\item$s_{\lambda/d/\mu}$ is cylindric Schur-positive, i.e.,
$$s_{\lambda/d/\mu}(\x)=\sum_{\nu\in P_{mn}, e\geq 0} c^{\lambda/d/\mu}_{\nu/e/\emptyset} s_{\nu/e/\emptyset}(\x),$$
with $c^{\lambda/d/\mu}_{\nu/e/\emptyset}\geq 0$. 

\item $c^{\lambda/d/\mu}_{\nu/e/\emptyset}$ is the same as $c^{\lambda/d-1/\mu}_{\nu/e-1/\emptyset}$ for all positive integers $e$, and $c^{\lambda/d/\mu}_{\nu/0/\emptyset}=C^{\lambda,d}_{\mu,\nu}$. 
\end{enumerate}
\end{thm}
To prove Theorem \ref{thm.main}, we first relate the cylindric Schur functions with affine Stanley symmetric functions.
\section{affine Stanley symmetric functions}
In this section, we review theory of affine Stanley symmetric functions and the affine nilCoxeter algebra. See \cite{lam.2006,lam.2008} for more details.
\subsection{Affine Stanley symmetric functions}
Let $\sn$ denote the affine symmetric group with simple generators $s_i$ for $i \in \bZ/n\bZ$ satisfying the relations
\begin{align*}
s_i^2&=1&\\
s_is_{i+1}s_i&=s_{i+1}s_is_{i+1}\\
s_is_j&=s_js_i&& \text{if } i-j\neq 1,-1.
\end{align*} 
where indices are taken modulo $n$. An element of the affine symmetric group may be written as a word in the generators $s_i$. A \emph{reduced word} of the element is a word of minimal length. The \emph{length} of $w$, denoted $\ell(w)$, is the number of generators in any reduced word of $w$. 

The subgroup of $\sn$ generated by $\{s_1,\cdots,s_{n-1}\}$ is naturally isomorphic to the symmetric group $S_n$. The \emph{$0$-Grassmannian elements} are minimal length coset representatives of $\sn/ S_n$. In other words, $w$ is $0$-Grassmannian if and only if all reduced words of $w$ end with $s_0$. More generally, for $i \in \mathbb{Z}/n\mathbb{Z}$ and $w \in \sn$, $w$ is called $i$-Grassmannian if all reduced words of $w$ end with $s_i$. We denote the set of $i$-Grassmannian elements by $\sn^i$. There is a weak order $<$ on $\sn$ defined by the covering relation $w\lessdot v$ fi $s_iw=v$ with $\ell(w)+1=\ell(v)$.

A word $s_{i_1}s_{i_2}\cdots s_{i_l}$ with indices in $\mathbb{Z}/n\mathbb{Z}$ is called \emph{cyclically decreasing} (resp. \emph{cyclically increasing}) if each letter occurs at most once and whenever $s_i$ and $s_{i+1}$ both occur in the word, $s_{i+1}$ precedes $s_i$ (resp. $s_i$ precedes $s_{i+1}$). For $J\varsubsetneq \mathbb{Z}/n\mathbb{Z}$, a \emph{cyclically decreasing element} $d_J$ (resp. a \emph{cyclically increasing element} $u_J$) is the unique cyclically decreasing permutation (resp. cyclically increasing permutation) in $\sn$ which uses exactly the simple generators in $\{s_j \mid j\in J\}$. For example, for $n=7$ and a subset $J=\{0,1,4,6\}$ of $\mathbb{Z}/7\mathbb{Z}=\{0,1,2,3,4,5,6\}$ we get $d_J=s_4s_1s_0s_6$ and $u_J=s_4s_6s_0s_1$ in $\widetilde{S}_7$. \\

For an element $w$ in $\sn$, we call $w_1w_2\cdots w_\ell$ a \emph{cyclically decreasing decomposition} of $w$ if $w_1w_2\cdots w_\ell =w$ satisfying $\ell(w_1)+\cdots+\ell(w_\ell)= \ell(w)$ and each $w_i$ is a cyclically decreasing element. Note that certain part $w_i$ can be an identity element. Now we are ready to define the affine Stanley symmetric functions.

\begin{definition}
For an element $w \in \sn$, an affine Stanley symmetric function $F_w$ is defined by
$$F_w(\x)= \sum  x_1^{\ell(w_1)}x_2^{\ell(w_2)}\cdots x_\ell^{\ell(w_\ell)}$$
where the sum runs over cyclically decreasing decompositions $w_1\cdots w_\ell$ of $w$.
\end{definition}

If $w$ is $0$-Grassmannian, we call $F_w$ an \emph{affine Schur function}. Let $\Lambda^{(n)}$ be the quotient of the ring of symmetric functions $\Lambda$ by monomial symmetric functions $m_\lambda$ for $\lambda_1\geq n$. Let $\Lambda_{(n)}$ be the subalgebra $\mathbb{Z}[h_1,\ldots,h_{n-1}]$ of $\Lambda$, where $h_i$ is a homogeneous symmetric functions of degree $i$. Let $\langle \cdot,\cdot \rangle_\Lambda$ be the Hall inner product of $\Lambda$. Then two algebras $\Lambda_{(n)}$ and $\Lambda^{(n)}$ are dual each other. 

Now we list known theorems about the affine Stanley symmetric functions.

\begin{thm} \cite{lam.2008,lapointe.morse.2007}
The affine Schur functions $\{ F_w \mid w \in \sn^0\}$ form a basis of $\Lambda^{(n)}$.
\end{thm}

\begin{thm}\cite{lam.2008}\label{thm.expansion}
The affine Stanley symmetric functions $F_w$ expand positively in terms of the affine Schur functions. Namely, we have
$$F_w= \sum_{v\in \sn^0} c^w_v F_v$$
for a non-negative integer $c^w_v$.
\end{thm}

Let $\{s^{(k)}_u \mid u \in \sn^0 \}$ be the dual basis of $\{F_u \mid u \in \sn^0\}$ of $\Lambda_{(n)}$ with respect to the induced Hall inner product $\langle \cdot, \cdot \rangle_\Lambda : \Lambda_{(n)} \times \Lambda^{(n)} \rightarrow \bZ$. Here, $k$ is equal to $n-1$ and will be equal throughout the paper. The symmetric function $s^{(k)}_u$ is called the \emph{$k$-Schur function}. It is known that the coefficients $c^w_v$ are the same as the structure constants of the $k$-Schur functions. We prove this fact by using the affine nilCoxeter algebra and noncommutative $k$-Schur functions in the next section.


\subsection{The affine nilCoxeter algebra}

The \emph{affine nilCoxeter algebra} $\mathbb{A}$ is the algebra generated by $A_0,A_1,\ldots,A_{n-1}$ over $\mathbb{Z}$, satisfying
\begin{align*}
A_i^2&=0&\\
A_iA_{i+1}A_i&=A_{i+1}A_i A_{i+1}\\
A_iA_j&=A_j A_i&& \text{if } i-j\neq 1,-1.
\end{align*} 
where the indices are taken modulo $n$. A subalgebra of $\mathbb{A}$ generated by $A_i$ for $i\neq 0$ is isomorphic to the nilCoxeter algebra studied by Fomin and Stanley \cite{fomin.stanley.1994}. The simple generators $A_i$ can be considered as the \emph{divided difference operators} in Kumar and Kostant's work. \\

The $A_i$ satisfy the same braid relations as the $s_i$ in $\sn$, i.e., $A_iA_{i+1}A_i=A_{i+1}A_iA_{i+1}$. Therefore it makes sense to define
\[ \begin{array} {rlll} A_w&=&A_{i_1}\cdots A_{i_l} & \mbox{where} \\ w&=&s_{i_1}\cdots s_{i_l} & \mbox{is a reduced decomposition.} \end{array} \]
One can check that 
\[A_vA_w=\left\{ \begin{array}{ll} A_{vw} & \mbox{if } \ell(vw)=\ell(v)+\ell(w) \\ 0 & \mbox{otherwise.}\end{array} \right. \]

Recall that $d_J$ is the cyclically decreasing element corresponding to a proper subset $J$ of $\znz$. For $i <n$, let
\[\h_i=\sum\limits_{\substack{ J\subset \mathbb{Z}/n\mathbb{Z} \\ |J|=i }} A_{d_J} \in \mathbb{A} \]
where $\h_0=1$ and $\h_i=0$ for $i< 0$ by convention. Lam \cite{lam.2006} showed that the elements $\{\h_i\}_{i<n}$ commute and freely generate a subalgebra $\mathbb{B}$ of $\mathbb{A}$, called the \emph{affine Fomin-Stanley algebra}. It is well-known that $\mathbb{B}$ is isomorphic to $\mathbb{Z}[h_1,\ldots,h_k]$ via the map sending $\h_i$ to $h_i$, where $h_i$ is a complete homogeneous symmetric function of degree $i$. Therefore, the set $\{\h_\lambda=\h_{\lambda_1}\ldots \h_{\lambda_l} \mid \lambda_1\leq k\}$ forms a basis of $\mathbb{B}$.\\

There is another basis of $\mathbb{B}$, called the \emph{noncommutative $k$-Schur functions} $\s_\lambda$, indexed by partitions $\lambda$ with $\lambda_1\leq k$. We call such a partition $\lambda$ a \emph{$k$-bounded partition}. There is a bijection between the set of $k$-bounded partitions and $\sn^0$ \cite{lapointe.morse.2005}. From now on, we also index the noncommutative $k$-Schur functions with $0$-Grassmannian elements using the bijection. We review the bijection between the set of $k$-bounded partitions and $\sn^0$ at the end of this section.\\

For an element $w \in \sn^0$, the noncommutative $k$-Schur function $\s_w$ is the image of $s^{(k)}_w$ via the isomorphism $\Lambda_{(k)}\cong \mathbb{B}$. Lam \cite{lam.2008} showed that the noncommutative $k$-Schur function $\s_w$ is the unique element in $\mathbb{B}$ that has the unique $0$-Grassmannian term $A_{w}$. The noncommutative $k$-Schur functions are non-equivariant version of the $j$ functions studied by Peterson \cite{peterson.1997} for affine type $A$. For details and the original definition of noncommutative $k$-Schur functions, see \cite{lam.2008,LLMSSZ}.\\

\begin{example}
For a positive integer $i<n$, we have $\s_{s_{i-1}s_{i-2}\ldots s_0} = \h_i$ since $\h_i$ is in $\mathbb{B}$ and $s_{i-1}s_{i-2}\ldots s_0$ is the unique cyclically decreasing element of length $i$ in $\sn^0$. Note that the $k$-bounded partition corresponding to $s_{i-1}s_{i-2}\ldots s_0$ is a partition $(i)$.
\end{example}

Lam showed the following theorem.
\begin{thm}\label{thm.kschur}\cite{lam.2008}
For $w \in \sn$, assume that we have an expansion
$$F_w = \sum_u c^w_u F_u$$
where the sum runs over $0$-Grassmannian elements $u$. Then we have
$$\s_u = \sum_{w \in \sn} c^w_u A_w$$

\end{thm}

The structure constants of $k$-Schur functions are the same as coefficients $c^w_u$ appearing in Theorem \ref{thm.kschur}. For $w \in \sn^0$, consider the coefficients of $A_w$ in the each side of the identity
$$\s_u \s_v =\sum _{\alpha \in \sn^0} d^\alpha_{u,v} \s_\alpha.$$
The coefficient of $A_w$ in the right-hand side is $d^w_{u,v}$. On the other hand, the coefficient of $A_w$ in the left-hand side is $\sum c^{u'}_u c^{v'}_v$ where the sum runs over elements $u' , v'$ in $\sn$ such that $u'v'=w$ with $\ell(u')+\ell(v')=\ell(w)$. Since $w$ is $0$-Grassmannian, $v'$ has to be $0$-Grassmannian as well. By the property of the noncommutative Schur functions $v'$ must be equal to $v$ and the summation becomes $c^{u'}_u$ where $u'$ satisfies $u'v=w$ with $\ell(u')+\ell(v')=\ell(w)$. Therefore, we have the following.

\begin{thm}\label{thm.cd}
For $w,u,u',v \in \sn$ satisfying $w=u'v$ with $\ell(w)=\ell(u')+\ell(v)$, we have
$c^{u'}_u=d^w_{u,v}$.
\end{thm}

Since the affine Fomin-Stanley algebra is commutative, we have the following symmetry of the coefficients $c^{u'}_u$.

\begin{corollary}\label{cor.symmetry}
For $\alpha=w_1 w_2 = v_1 v_2$ in $\sn^0$ such that $\ell(w_1)=\ell(v_2)$, $\ell(w_2)=\ell(v_1)$, and $\ell(\alpha)=\ell(w_1)+\ell(w_2)=\ell(v_1)+\ell(v_2)$, we have $c^{w_1}_{v_2} = c^{v_1}_{w_2}$. 
\end{corollary}
\begin{proof}
By Theorem \ref{thm.cd}, we have $c^{w_1}_{v_2}=d^{\alpha}_{v_2,w_2}=d^\alpha_{w_2,v_2}=c^{v_1}_{w_2}$.
\end{proof}

Now we review Denton's work regarding affine permutations to explain the bijection between the set of $k$-bounded partitions and $\sn^0$. For $w\in \sn$, consider the unique maximal $J\varsubsetneq \mathbb{Z}/n\mathbb{Z}$ such that $w=v d_J$ with $\ell(w)=\ell(v)+\ell(d_J)$ for some $u \in \sn$. Denton showed the such $d_J$ is unique in the following sense:

\begin{lemma} \cite[Corollary 18]{denton.2012}\label{denton.lemma}
For given an affine permutation $w$, there is a unique subset $J$ of $\znz$ so that whenever we have $w=u d_{J'}$ for some $u$ and a cyclically decreasing element $d_{J'}$ with $\ell(w)=\ell(u)+\ell(d_{J'})$, $J$ contains $J'$.
\end{lemma}

We denote such a set $J$ by $D(w)$ and $|J|$ by $\maxr(w)$ Similarly, consider the unique maximal set $J'\varsubsetneq \mathbb{Z}/n\mathbb{Z}$ such that $w=v u_{J'}$ with $\ell(w)=\ell(v)+\ell(u_{J'})$ for some $v \in \sn$ and denote $J'$ (resp. $|J'|$) by $U(w)$ (resp. $\maxc(w)$). Note that by repeating Lemma \ref{denton.lemma}, we get the following corollary.

\begin{corollary}\cite[Corollary 20]{denton.2012}\label{denton.lemma2} 
Every affine permutation $w$ has a unique maximal decomposition into cyclically decreasing elements, i.e., $w=d_{J_p}\cdots d_{J_1}$ where $\ell(w)=\sum_{i=1}^p |J_i|$ and the sequence $(|J_1|,|J_2|,\ldots,|J_p|,0,0,\ldots)$ is the maximum with respect to the lexicographic order.
\end{corollary}

Denton also showed that a sequence $(|J_1|,|J_2|,\ldots,|J_p|,0,0,\ldots)$ forms a partition such that the first part is at most $k$, which is a $k$-bounded partition. When $w$ is $0$-Grassmannian, the map sending $w$ to a $k$-bounded partition $$\lambda=(|J_1|,|J_2|,\ldots,|J_p|)$$ becomes a bijection. Each set $J_i$ is determined by the size $|J_i|=\lambda_i$ ,by $J_i=[-j+1,\lambda_j-j]$ \cite[Corollary 39]{denton.2012}.

\section{Combinatorics on $321$-avoiding permutations}

An element $w\in \sn$ is called $321$\emph{-avoiding} if any reduced word of $w$ does not contain a subword $s_is_{i+1}s_i$ for any $i \in I$. Lam \cite{lam.2006} showed that a cylindric skew Schur function is an affine Stanley symmetric function for some $w\in \sn$. In this section, we develop combinatorics on $321$-avoiding affine permutations and reprove and generalize the above result more systematically. Then we use these results to prove Theorem \ref{thm.main}.\\

If $\lambda[r]/\mu[s]$ is a cylindric shape of type $(m,n)$, let $s_i\cdot (\lambda[r]/\mu[s])$ be a cylindric shape $\lambda'[r']/\mu[s]$ of type $(m,n)$ such that $D_{\lambda'[r']/\mu[s]} = D_{\lambda[r]/\mu[s]} \cup \{ \langle p,q \rangle \mid q-p=i \}$ for some $p,q \in \mathbb{Z}$ if such $p$ and $q$ exists, and $s_i\cdot (\lambda[r]/\mu[s])$ is not defined otherwise. If $\alpha=(\ldots, \alpha_{-1},\alpha_0,\alpha_1,\ldots)$ is the $(m,n)$-periodic increasing bi-infinite sequence corresponding to $\lambda[r]$, then $s_i\cdot (\lambda[r]/\mu[s])$ is well-defined if and only if there is an integer $p$ such that $\alpha_p+1-p= i$ modulo $n$ and $\alpha_{p}<\alpha_{p+1}$. Visually, one can obtain $D_{\lambda'[r']}$ from $D_{\lambda[r]}$ by adding a box at $i$-th diagonal when possible. In this case, $s_i\cdot (\lambda[r]/\mu[s])$ is $\lambda'[r']/\mu[s]$ where the $(m,n)$-periodic increasing bi-infinite sequence corresponding to $\lambda'[r']$ is $\beta$ where $\beta_j=\alpha_j + 1$ if $j=p$ modulo $m$ and $\alpha_j$ otherwise. Hence we also define an action $s_i$ on the set of $(m,n)$-periodic increasing sequences similarly. In this paper, bi-infinite sequences are always $(m,n)$-periodic and increasing.\\

For an element $A_i$ in $\mathbb{A}$, one can define an action on the set of cylindric shapes of type $(m,n)$ by
$$A_i \cdot (\lambda[r]/\mu[s]) = \begin{cases} s_i\cdot (\lambda[r]/\mu[s]) & \text{if } s_i\cdot (\lambda[r]/\mu[s]) \text{ is well-defined} \\ 0 & \text{otherwise.} \end{cases}$$

Note that this gives an action of $\mathbb{A}$ on the set $C_{mn}$ of cylindric shapes of type $(m,n)$. In fact, the following is true.

\begin{lemma}\label{lemma.ai}
As an action on the set $C_{mn}$ of cylindric shapes of type $(m,n)$ (or the set of $(m,n)$-periodic bi-infinite sequences), the generators $A_i$ satisfy the following.
\begin{enumerate}
\item $A_i^2=0$,
\item $A_iA_j=A_jA_i \quad$ if $i-j \neq \pm 1$,
\item $A_iA_{i+1}A_i=A_{i+1}A_iA_{i+1}= 0$,
\item For a cyclically decreasing element $w$ of length $>n-m$, $A_w =0$,

\item For a cyclically increasing element $w$ of length $>m$, $A_w =0$.
\end{enumerate}
\begin{proof}
Note that it is enough to show that $A_i$'s satisfy the above relations as an action on the set of $(m,n)$-periodic bi-infinite sequences. 

$(1)$ is obvious. To prove $(2)$, for distinct elements $i,j \in \mathbb{Z}/n\mathbb{Z}$ satisfying $i-j\neq \pm 1$, note that $A_iA_j \cdot \lambda[r]$ is not $0$ if and only if one can add two boxes at diagonal $i$ and $j$ and in this case $A_i A_j \cdot \lambda[r]$ is obtained by adding such two boxes in $\cm$. Since $j-i\neq \pm$, such two boxes cannot be adjacent and we have $A_i A_j \cdot \lambda[r] = A_j A_i \cdot\lambda[r]$. For $(3)$, assume that $A_iA_{i+1}A_i \cdot \alpha$ is $\beta$ for some $(m,n)$-periodic sequence $\beta$. Then $D_{\beta/\alpha}$ has the cardinality $3$ and has two distinct boxes $\langle a_1,a_2\rangle ,\langle b_1,b_2 \rangle$ at $i$-th diagonal. Since there is only one $i$-th diagonal in $\cm$, without lose of generality one can assume that $\langle a_1+1,a_2+1\rangle =\langle b_1,b_2\rangle$. Since $D_{\beta/\alpha}$ is a cylindric diagram, both $\langle a_1+1,a_2\rangle$ and $\langle a_1,a_2+1\rangle $ are in $D_{\beta/\alpha}$. Since each box lies in $(i-1)$-th and $(i+1)$-th diagonals respectively, $A_i A_{i+1}A_i \cdot\alpha$ cannot be $\beta$. Also one can use the same argument to show that $A_iA_{i-1}A_i\cdot \alpha=0$.\\

To prove $(4)$, assume that for a cyclically decreasing element $w$, we have $A_w \cdot \alpha = \beta$ for some $(m,n)$-periodic sequences $\alpha,\beta$. Then one can show that $D_{\beta/\alpha}$ is a \emph{horizontal strip}, i.e., there are no two boxes at the same column. Indeed, if there are two adjacent boxes at the same column then two boxes are at $i$-th diagonal and $(i-1)$-th diagonal respectively for some $i \in \bZ/n \bZ$. In any reduced word of $w$, $s_i$ precedes $s_{i-1}$ so that $w$ is not a cyclically decreasing element. Note that for any cylindric diagram in $\cm$ there are at most $n-m$ columns so that $\ell(w)\leq n-m$, proving $(4)$. Proof of $(5)$ is similar to the proof of $(4)$.
\end{proof}
\end{lemma}

Due to Lemma \ref{lemma.ai}, we define the quotient algebra $\am$ of $\mathbb{A}$ by the relations $(3)-(5)$ in Lemma \ref{lemma.ai}. One can rephrase the conditions $(4)$-$(5)$ using the following lemma.

\begin{lemma}\label{maxr}
For $w \in \sn$ and a positive integer $q<n$, $\maxr(w)> q$ if and only if there exist $v_1,v_2 \in \sn$ and $S\varsubsetneq \bZ/n\bZ$ satisfying $w= v_1 d_S v_2$ and $\ell(w)=\ell(v_1)+\ell(v_2)+|S|$ and $|S|>q$. Similarly, $\maxc(w)>q$ if and only if $w= v_1 u_S v_2$ for some $S\varsubsetneq \bZ/n\bZ$ with $\ell(w)=\ell(v_1)+\ell(v_2)+|S|$ and $|S|>q$.
\end{lemma}
\begin{proof}
'only if' statement is obvious, since there exist an affine permutation $u$ and $S\subset \znz$ satisfying $w=ud_S$ with $\ell(w)=\ell(u)+|S|$ and $|S|=\maxr(w)$ by the definition of $\maxr(w)$. To proof 'if' statement, assume that we have $w=v_1 d_S v_2$ satisfying conditions in the above lemma. First we prove that there exist an element $u \in \sn$ and $S'\varsubsetneq \bZ/n\bZ$ satisfying $ d_S v_2 = u d_{S'}$ with $\ell(v_2)=\ell(u)$ and $|S'|=|S|$. Let $\ell$ be $\ell(v_2)$ and $p$ be $|S|$, and consider the identity $(\s_{s_0})^{\ell} \h_p = \h_p (\s_{s_0})^{\ell}$. The coefficient of $A_{d_S v_2}=A_{d_S}A_{v_2}$ in $\h_p(\s_{s_0})^{\ell} $ is at least 1, since the coefficient of $A_{v_2}$ at $(\s_{s_0})^\ell$ is positive (it is equal to the number of reduced words of $v_2$) and the coefficient of $A_{d_S}$ at $\h_p$ is one. Therefore, the coefficient of $A_{d_Sv_2}$ in $ (\s_{s_0})^{\ell}\h_p$ is positive. Since every term in $(\s_{s_0})^{\ell}\h_p $ is of the form $A_u d_{S'}$ for some affine permutation $u \in \sn$ with $\ell(u)=\ell$ and $S'\varsubsetneq \bZ/n\bZ$ with $|S'|=p$, there are such $u$ and $S'$ satisfying $d_Sv_2=ud_{S'}$ and we proved the claim. By the claim, we have $w=v_1 u d_{S'}$ with $\ell(w)=\ell(v_1)+\ell(u)+|S'|$ and $\maxr(w)\geq |S'|>q$ by the definition of $\maxr$.
\end{proof}
By Lemma \ref{maxr},  $\am$ can also be defined by the quotient of $\mathbb{A}$ by the relation $(1)$-$(3)$ and the relation
$$A_w=0 \quad \text{if } \maxr(w)>n-m \text{ or } \maxc(w)>m.$$
We abuse a notation that the image of $A_i$ under the projection $\mathbb{A}\rightarrow \am$ is denoted by the same $A_i$.
Therefore, the set $\{A_w \mid w\in \sn \text{ is } 321\text{-avoiding},\maxc(w) \leq m, \maxr(w) \leq n-m \}$ forms a basis of $\am$. We denote the set $\{w \mid w\in \sn \text{ is } 321\text{-avoiding},\maxc(w) \leq m, \maxr(w) \leq n-m \}$ by $\anm$. We denote the set $\sn^0 \cap \anm$ by $\anmz$.\\

Let $\bm$ be the subalgebra of $\am$ generated by $\bh_1,\ldots, \bh_{n-m}$, where $\bh_i$ is the image of $\h_i$ via the projection $\mathbb{A}\rightarrow \am$. For $i<n$, let ${\bf e}_i$ be an element in $\mathbb{A}$ defined by $\sum_w A_w$ where the sum runs over cyclically increasing elements $w$ of length $i$. Let ${\bf \overline{e}}_i$ be the image of ${\bf e}_i$ via the projection $\mathbb{A}\rightarrow \am$. For $w \in \sn^0$, let $\bs_w$ be the image of $\s_w$ via the projection $\mathbb{A}\rightarrow \am$. We show the following lemmas for later purposes.

\begin{lemma}\label{lemma.ij}
For $J\varsubsetneq \bZ/n\bZ$, $A_{d_J}A_i$ is zero in $\am$ if $i \in J$. Similarly, $A_i A_{d_J}$, $A_{u_J}A_i$ and $A_i A_{u_J}$ are zero in $\am$ if $i\in J$.
\end{lemma}
\begin{proof}
Let $[p,q]$ be the maximal interval containing $i$ and contained in $J$. Here $[p,q] =\{p,p+1,\ldots,q-1,q\}$ modulo $n$. It is enough to show that $A_q A_{q-1}\ldots A_{p+1}A_p A_i$ is zero in $\am$. If $p=i$ then it is clearly zero so assume that $p\neq i$. Since $i$ is contained in $[p,q]=[p,i-2]\cup [i-1,q]$, we have
$$A_q A_{q-1}\ldots A_{p+1}A_p A_i=(A_q \ldots A_{i+1} A_i A_{i-1}) A_i (A_{i-2} A_{i-3}\ldots A_{p+1}A_p)$$
so this element contains a word $A_i A_{i-1}A_i$ hence it is zero in $\am$. The proof of the second statement is similar, hence omitted.
\end{proof}
\begin{remark}\label{remark.ij}
Note that the proof of Lemma \ref{lemma.ij} only uses the relations $(1)$-$(3)$, which uses the extra condition $321$-avoiding.
\end{remark}

\begin{lemma}\label{lemma.ribbon}
For $J,J' \in \znz$ such that $J\cap J' $ is nonempty, $A_{e_J}A_{d_{J'}}$ is zero in $\am$. Similarly, $A_{d_{J'}}A_{e_J}$ is also zero in $\am$.
\end{lemma}
\begin{proof}
Assume that $\ell(e_J)+\ell(d_{J'})=\ell(e_J d_{J'})$ and $e_J d_{J'}$ is $321$-avoiding, otherwise $A_{e_J}A_{d_{J'}}$ is zero in $\am$. Choose an element $i$ in $J\cap J'$. Since $s_i$ appears twice in any reduced word of $e_J d_{J'}$ and $e_J d_{J'}$ is $321$-avoiding, both $s_{i-1}$ and $s_{i+1}$ appears between two $s_i$'s in the reduced word. However, $i-1$ cannot be in $J$ since otherwise $s_i$ appeared left of $s_{i-1}$ precedes $s_{i-1}$ and it contradicts with the fact that $e_J$ is cyclically decreasing. Also, $i-1$ cannot be in $J$ since otherwise $s_{i-1}$ precedes the $s_i$ that appears right of this $s_{i-1}$. This makes a contradiction and we are done. One can prove the second statement similarly.
\end{proof}
An element $w\in \sn$ is called $n$\emph{-connected ribbon} if $w$ is of the form $u_{J^c}d_{J}$ for some $J\varsubsetneq \mathbb{Z}/n\mathbb{Z}$.
For such a $w$, if $A_w$ is nonzero in $\am$, then $|J^c|\leq m$ and $|J|\leq n-m$ by Lemma \ref{maxr} so the inequalities become the equality.\\

For the rest of the section, we study combinatorics of $n$-connected ribbon to prove Proposition \ref{prop.decomp}. Then we prove Theorem \ref{thm.main}.

\begin{corollary}\label{cor.ribbon}
If $w_1, w_2$ are $n$-connected ribbons such that $\ell(w_1w_2)=\ell(w_1)+\ell(w_2)$ and $A_{w_1w_2}$ is nonzero in $\am$, then $w_1=w_2$.
\end{corollary}
\begin{proof}
Since $w_i$ is a $n$-connected ribbon, $w_i= e_{J_i^c}d_{J_i}$ for some $J_i \varsubsetneq \znz$. Since $e_{J_1^c}d_{J_1}e_{J_2^c}d_{J_2}$ is not reduced, $\ell(d_{J_1}e_{J_2^c})=\ell(d_{J_1})+\ell(e_{J_2^c})$. Note that since $A_{w_1}A_{w_2}$ is nonzero in $\am$, $|J_1|=|J_2|=n-m$ and $|J_2^c|=m$. Therefore, by Lemma \ref{lemma.ribbon}, $J_1$ and $J_2^c$ does not intersect and we get $J_1=J_2$ and $w_1=w_2$.
\end{proof}
Let $r_m$ be $u_{[-m,-1]}d_{[0,n-m-1]}$, which is
$$s_{-m}s_{-m+1}\cdots s_{-1} s_{n-m-1}s_{n-m-2}\cdots s_1 s_0.$$
We show that $r_m$ is the unique $n$-connected ribbon which is $0$-Grassmannian and $A_{r_m}$ is nonzero in $\am$. In fact, we show a stronger statement.
\begin{thm} \label{thm.kschurrm} As an element in $\am$, we have
$$ \bs_{r_m}= \sum_v A_v$$
where $v$ runs over all $n$-connected ribbon $u_{J^c}d_J$ for some $J\varsubsetneq \bZ/n\bZ$ with $|J|=n-m$. Therefore,  $r_m$ is the unique $n$-connected ribbon which is $0$-Grassmannian and $A_{r_m}$ is nonzero in $\am$. 
\end{thm}

\begin{proof}
First we show that $\bs_{r_m}= \overline{\bf e}_m\overline{\h}_{n-m}$. It is enough to show that for a $0$-Grassmannian element $v$, $A_v$ appears in $\overline{\bf e}_m\overline{\h}_{n-m}$ if and only if $v=r_m$. Let $v=u_{J'} d_{J}$ for some $J,J' \varsubsetneq \bZ/n\bZ$ with $|J|=n-m, |J'|=m$. Since $v$ is $0$-Grassmannian, $d_J$ has to be $s_{n-m-1}\ldots s_1s_0$. 
By Lemma \ref{lemma.ribbon}, $J'$ has to be $J^c$.

Every (nonzero) element $A_x$ appearing in $\overline{\bf e}_m\overline{\h}_{n-m}$ is of the form $u_{K_1}d_{K_2}$ for some $K_1,K_2 \in \znz$ with $|K_1|=m$ and $|K_2|=n-m$. Since $A_x$ is not zero in $\am$, $K_1=K_2^c$ by Lemma \ref{lemma.ribbon} and we are done.
\end{proof}

\begin{lemma}\label{lemma.wm1}
For $v \in \sn$ such that $v r_m$ is $321$-avoiding and $\ell(v r_m)=\ell(v)+\ell(r_m)$, then $v$ is $0$-Grassmannian.
\end{lemma}
\begin{proof}
If $v=v's_i$ for some nonzero $i$ then it suffices to show that $\ell(s_i r_m)<\ell(r_m)$. Assume that we have $\ell(s_i r_m)>\ell(r_m)$. Since $r_m= u_{[-m,-1]} d_{[0,n-m-1]}$, $i$ is not $-m$. If $i$ is in $[-m+1,-1]$ then $s_i u_{[-m,-1]}$ is not $321$-avoiding by Remark \ref{remark.ij}. The last case is when $i$ is in $[1,n-m-1]$. If $i$ is $n-m-1$ then $s_i u_{[-m,-1]} s_{n-m-1}$ is $s_{n-m-1} s_{n-m} s_{n-m-1} u_{[-m+1,-1]}$ so it is not $321$-avoiding. If $i$ is in $[1,n-m-2]$ then $s_i u_{[-m,-1]} d_{[0,n-m-1]}$ is 
$u_{[-m,-1]} s_i d_{[0,n-m-1]}$ 
and $s_i d_{[0,n-m-1]}$ is not $321$-avoiding by Remark \ref{remark.ij}, which is a contradiction.
\end{proof}
\begin{lemma} \label{lemma.wm2}
For $w \in \sn$, assume that for any $i \in \znz$ the generator $s_i$ occurs in some reduced word of $w$. Also assume that $A_w$ is nonzero in $\am$. Then there exist unique $v, u \in  \sn$ such that $w=v u$, $\ell(w)=\ell(v)+\ell(u)$ and for $i \in \znz$, $s_i$ occurs in a reduced word of $u$ exactly once. Similarly, assuming the same condition on $w$, there exist unique $u',v' \in \sn$ such that $w=u'v'$, $\ell(w)=\ell(u')+\ell(v')$ and for $i\in \znz$, $s_i$ occurs in a reduced word of $u'$ exactly once.
\end{lemma}
\begin{proof}
It is enough to prove the first statement by the symmetry. First we show the existence of $u$ and $v$. Since $A_w$ is nonzero in $\am$, then $\maxr(w)\leq n-m$ and $\maxc(w) \leq m$ by Lemma \ref{maxr}. Assume that $w= v_1 v_2$ for some $v_1,v_2 \in \sn$ satisfying $\ell(w)=\ell(v_1)+\ell(v_2)$ and the generator $s_i$ occurs in a reduced word of $v_2$ at most once. Choose $v_1,v_2$ so that the length of $v_2$ is maximal.

If $\ell(v_2)$ is equal to $n$ then the lemma is proved, so assume that $\ell(v_2)<n$. Choose an interval $[p,q] \subset \znz$ such that for any $i \in [p,q]$, $s_i$ does not occur in a reduced word of $v_2$. Choose a maximal $[p,q]$ so that $s_{p-1}$ and $s_{q+1}$ does appear in a reduced word of $v_2$.

Consider a reduced word $s_{i_1}\cdots s_{i_\ell}$ of $v_1$ so that the maximum of the set $\{ j \mid i_j \in [q+1,p-1] \}$ is the largest and denote the maximum by $j$. If $i_{a}$ is equal to $p$ for some $a>j$, since $v_1v_2$ is $321$-avoiding and a reduced word of $v_2$ contains $s_p$, $s_{p-1}$ should appear in $s_{i_a}\cdots s_{i_\ell}$. However, since $a>j$ it is a contradiction and $i_a$ cannot be equal to $p$ for any $a>j$. By the similar argument, $i_a$ cannot be equal to $q$ for any $a>j$. 

Also note that by the definition of $j$, $i_a$ is not in $[q+1,p-1]$ for all $a>j$. So $i_a$ is in $[p+1,q-1]$ and $s_{i_a}$ commutes with $s_{j}$. If $j$ is not $\ell=\ell(v_1)$, then since $s_{i_1}\cdots s_{i_\ell} = s_{i_1}\cdots s_{i_{j-1}} s_{i_{j+1}} s_{i_j} s_{i_{j+2}} \cdots s_{i_\ell}$, it contradicts with the definition of $j$. Therefore, we have $j=\ell$. Then by setting $w_1=v_1 s_\ell$ and $w_2= s_\ell v_2$, we have $w=w_1w_2$ with $\ell(w)=\ell(w_1)+\ell(w_2)$ and the generator $s_i$ occurs in a reduced word of $w_2$ at most once for $i\in \znz$. Since $\ell(w_2)=\ell(v_2)+1$ it contradicts with the assumption on $v_2$.\\

To prove the uniqueness, first we write $w=vu=v e_{J_c}d_J$ for some $J\subset \znz$. Since $w\in \anm$, $|J|=n-m$ and by Corollary \ref{denton.lemma2}, the set $J$ is uniquely determined. Therefore, $u=e_{J_c}d_J$ is also uniquely determined.
\end{proof}

\begin{lemma}\label{lemma.ribbon2}
For $w\in \anm$ and for any $i\in \znz$ $s_i$ occurs in a reduced word of $w$ exactly once then $w$ is a $n$-connected ribbon.
\end{lemma}
\begin{proof}
First we write $w= v d_J$ for $v\in \sn$ and $J \subset \znz$ with maximal $J$. Assume that we have an interval $[p,q] \subset \znz$ so that the interval $[p,q]$ and $J$ does not intersect. Choose a maximal interval $[p,q]$ so that $p-1,q+1 \in J$. Since $s_i$ does not appear in $d_J$ for $i \in [p,q]$, it occurs in a reduced word of $v$. Since $s_{p-1}$ and $s_{q+1}$ does not appear in a reduced word of $v$, one can write $v=v_1v_2$ so that $s_i$ occurs in a reduced word of $v_2$ if and only if $i\in [p,q]$. If $v_2$ is not a cyclically increasing element, there exists $i\in [p,q-1]$ so that $s_{i+1}$ precedes $s_i$. Choose a minimal possible $i$, then both $s_{i-1}$ and $s_{i+1}$ precede $s_i$ or $i=p$, so that one can write $v_2= v_3 s_{i}$ with $\ell(v_2)=\ell(v_3)+1$. Since $i$ is in $[p,q-1]$, $s_i d_J$ is cyclically decreasing, it makes a contradiction by the definition of $J$ and $i+1\notin J$. By applying this technique to all possible $[p,q]$ such that $[p,q]$ does not intersect with $J$, we prove that $v$ is a cyclically increasing element. By Lemma \ref{lemma.ribbon} and the fact that $\ell(w)=n$, $v$ is equal to $u_{J^c}$ and $w=u_{J^c}d_J$. Therefore, $w$ is a $n$-connected ribbon.
\end{proof}
\begin{corollary}\label{cor.wm2}
For $w\in \anm$, let $d$ be the minimum of the number of times $s_{i}$ occurs in a reduced word of $w$ for $i \in \znz$. Then we have $w=w^{(0)} v^d$ for $w^{(0)} \in \anm$ such that $s_{i}$ does not occur in a reduced word of $w^{(0)}$ for some $i$ satisfying $\ell(w)=\ell(w^{(0)})+nd$, and $v$ a connected $n$-ribbon. We call the decomposition $w=w^{(0)} v^d$ a \emph{ribbon decomposition}. Moreover, if $w$ is $0$-Grassmannian, then $v$ is equal to $r_m$ and $w^{(0)}$ is also $0$-Grassmannian.
\end{corollary}
\begin{proof}
First note that the number of $s_i$ that appears in a reduced word of $w$ for given $i$ is independent on a reduced word, since $w$ is 321-avoiding. By recursively applying Lemma \ref{lemma.wm2}, we have a decomposition $w=w^{(0)} v^d$ for some connected $n$-ribbon by Corollary \ref{cor.ribbon}. If $w$ is $0$-Grassmannian, then $v$ is also $0$-Grassmannian and hence equal to $r_m$. By Lemma \ref{lemma.wm1}, $w^{(0)}$ is also $0$-Grassmannian

\end{proof}

\begin{lemma}\label{lemma.d}
For an element $w$ in $\anmz$ and a ribbon decomposition $w=w^{(0)}v^d$, $s_{n-m}$ does not appear in a reduced word of $w^{(0)}$. Therefore, $d$ is the number of $s_{n-m}$ appears in a reduced word of $w$.
\end{lemma}
\begin{proof} Assume that $s_{n-m}$ appears in a reduced word of $w^{(0)}$. We first claim that either $\maxc(w^{(0)})<m$ or $\maxr(w^{(0)})<n-m$. If $\maxc(w^{(0)})=m$, then $w^{(0)}$ can be written as $w_1u_J$ where $w_1\in \anm$ and $J$ is a proper subset of $\znz$ satisfying $\ell(w^{(0)})=\ell(w_1)+\ell(u_J)$. Since $w^{(0)}$ is $0$-Grassmannian, $J$ must be an interval ${[-m+1,0]}$. On the other hand, if $\maxr(w^{(0)})=n-m$, $w^{(0)}$ can be written as $w_2d_{[0,n-m-1]}$ for $w_2\in \anm$ satisfying $\ell(w^{(0)})=\ell(w_2)+n-m$. Since for any $i \in\znz$ which is not $n-m$, $i$ appears either in $d_{[-m+1,0]}$ or $d_{[0,n-m-1]}$ it makes a contradiction. Therefore, without lose of generality, we may assume that $\maxc(w^{(0)})<m$.
Let $w=d_{J_p}\cdots d_{J_1}$ be the unique maximal decomposition into cyclically decreasing elements. Since $w$ is $0$-Grassmannian, $J_j$ must be $[-j+1,\lambda_j-j]$ for any $1 \leq j \leq p$ for some integer sequences $n-m\geq \lambda_1\geq \lambda_2\geq \ldots \geq \lambda_p\geq 0$, which we mentioned the proof at the end of Section 3 except that now $\lambda_1$ is less than equal to $n-m$. Since $\maxr(w^{(0)})<m$, it follows that $p<m$ and any interval $[-j+1,\lambda_j-j]$ does not contain $n-m$ since $-m<-p+1\leq -j+1$ and $\lambda_j-j\leq \lambda_1-1 \leq n-m-1$. This makes a contradiction to our assumption. \end{proof}
\begin{thm}\label{thm.bij}
There is a bijection $\phi$ between the set $\anmz$ and the set of the cylindric shapes $\lambda/d/\emptyset$ for $d\geq0$. Moreover, for $w,v \in \anmz$ $D_{\phi(v)}$ contains $D_{\phi(w)}$ if and only if $w< v$ where $<$ is the weak order in $\sn$. Also, the affine Schur function $F_w$ is the same as the cylindric Schur functions $s_{\phi(w)}$.
\end{thm}
\begin{proof}
For an element $w \in \anmz$, we define $\phi(w)$ by $A_w \cdot \emptyset/0/\emptyset$ and we will show that $\phi$ is a bijection. Assume that we have a ribbon decomposition $w = w^{(0)} r_m^d$ for some $w^{(0)} \in \anmz$.
When $d=0$, then the bijectivity of $\phi$ comes from the well-known bijection $\phi'$ between $P_{mn}$ and the subset of $S_n$ with unique descent at $m$. Indeed, $\phi$ is $ \psi \circ \phi'$ where $\psi :S_n \rightarrow \sn$ is defined by sending $s_i$ to $s_{i-m}$. Note that the image of $\psi$ does not contain an element whose reduced word contains $s_{n-m}$.\\

 If $d>0$, we have $w = w^{(0)} r_m^d$ for some $w^{(0)} \in \anmz$ such that $s_{n-m}$ is not used in any reduced word of $w^{(0)}$ and $\ell(w)=\ell(w^{(0)})+nd$. Note that the decomposition $w=w^{(0)} r_m^d$ is unique. Then $\phi(w)$ is $\phi(w^{(0)})/d/\emptyset$ and $\phi$ becomes the bijection between the set $\anmz$ and the set of cylindric shapes $\lambda/d/\emptyset$ for $\lambda\in P_{mn}$ and $d\geq 0$. \\

Note that by the construction of the bijection $\phi$, $s_iw=v$ for some $w \lessdot v$ if and only if $A_i \cdot \phi(w)= A_i A_w \cdot (\emptyset/0/\emptyset)=\phi(v)$. Note that $A_i \cdot \phi(w) = \phi(v)$ if and only if $D_{\phi(v)}$ contains $D_{\phi(w)}$ and the difference $D_{\phi(v)}/D_{\phi(w)}$ consists of a box at $i$-th diagonal. This proves the second statement.\\

For the last statement, recall that for $v \in \sn$ such that $A_v$ is nonzero and cylindric diagrams $D,D'$ such that $A_v\cdot D= D'$, $D'/D$ is a horizontal strip if and only if $v$ is cyclically decreasing. Therefore, a cyclically decreasing decomposition of $w$ corresponds to a semi-standard tableau on a cylindric diagram $D_{\phi(w)}$ with the same weight, so $F_w$ is equal to the cylindric Schur functions $s_{\phi(w)}$.
\end{proof}

\begin{corollary}\label{cor.skew}
For a cylindric shape $\lambda/d/\mu$ of type $(m,n)$, a cylindric skew Schur function $s_{\lambda/d/\mu}$ is the same as $F_w$ where $w =\phi^{-1}(\lambda/d/\emptyset )\left( \phi^{-1}(\mu/0/\emptyset)\right)^{-1} \in\anm$.
\end{corollary}
\begin{proof}
Since $\lambda/d/\mu$ is a cylindric shape of type $(m,n)$, $D_{\lambda[d]}$ contains $D_{\mu[0]}$. By Theorem \ref{thm.bij}, there is $w \in \sn$ such that $A_w \cdot (\mu/0/\emptyset)= \lambda/d/\emptyset$ and we have $w\phi^{-1}(\mu/0/\emptyset) = \phi^{-1} (\lambda/d/\emptyset)$ such that $\ell(w)= |D_{\lambda[d]}/D_{\mu}|=|D_{\lambda/d/\mu}|$. Since $A_w \cdot (\mu/0/\emptyset)$ is not zero, $w$ is in $\anm$. Also since $A_w \cdot (\mu/0/\emptyset)= \lambda/d/\emptyset$, we have $w =\phi^{-1}(\lambda/d/\mu )\left( \phi^{-1}(\mu/0/\emptyset)\right)^{-1}$.\\

Recall that for $v \in \anm$ and cylindric diagrams $D,D'$ such that $A_v\cdot D= D'$, $D'/D$ is a horizontal strip if and only if $v$ is cyclically decreasing. Therefore, a cyclically decreasing decomposition of $w$ corresponds to a semi-standard tableau on a cylindric diagram $D_{\lambda/d/\mu}$ with the same weight, so $F_w$ is equal to the cylindric Schur functions $s_{\lambda/d/\mu}$.

\end{proof}
\begin{proposition}\label{prop.decomp}
For $w\in \anmz$ and for a ribbon decomposition $w=w^{(0)} r_m^d$, we have
$$\bs_{w}= \bs_{w^{(0)}} (\bs_{r_m})^d$$
in $\bm$.
\end{proposition} 
\begin{proof}
We use an induction on $d$. Note that for $v \in \anmz$, $A_v$ is the unique $0$-Grassmannian element in $\bs_{v}$, so it is enough to show that $A_v$ appears in the expansion of $\bs_{w^{(0)}} (\bs_{r_m})^d$ for a $0$-Grassmannian element $v$ if and only if $v=w$. Since $A_v$ appears in $\left(\bs_{w^{(0)}} (\bs_{r_m})^{d-1}\right)\bs_{r_m}$ there exists $v_1 \in \sn$ such that $v=v_1 v_2$ with $\ell(v)=\ell(v_1)+\ell(v_2)$ and $A_{v_1}$ is in the expansion of $\bs_{w^{(0)}} (\bs_{r_m})^{d-1}$ and $A_{v_2}$ appears in $\bs_{r_m}$. Since $v_2$ is also $0$-Grassmannian, $v_2$ must be $r_m$. By Lemma \ref{lemma.wm1}, $v_1$ is $0$-Grassmannian. By an induction, we have $v_1=w^{(0)}r_m^{d-1}$ and $v=w^{(0)}r_m^{d}=w$. \end{proof}

Now we are ready to prove Theorem \ref{thm.main}.\\

{\it Proof of Theorem \ref{thm.main}.} By Corollary \ref{cor.skew}, $s_{\lambda/d/\mu}$ is the same as $F_w$ where $w =\phi^{-1}(\lambda/d/\emptyset )\left( \phi^{-1}(\mu/0/\emptyset)\right)^{-1}$ is an element in $\anm$. Let $\alpha$ denote $\phi^{-1}(\lambda/d/\emptyset)$ and let  $v$ denote $\phi^{-1}(\mu/0/\emptyset)$ so that we have $wv=\alpha$ satisfying $\ell(w)+\ell(v)=\ell(\alpha)$. Note that both $\alpha$ and $v$ are in $\anmz$.
By Theorem \ref{thm.expansion}, we have the expansion 
$$F_w= \sum_{u \in \sn} c^w_u F_u.$$
We first show that if $c^w_u$ is nonzero, then $u$ is in $\anmz$. By Theorem \ref{thm.cd}, $c^w_u$ is equal to $d^\alpha_{u,v}$ by comparing the coefficient of $A_w$ in the both side of an equality $\s_u \s_v =\sum_{\alpha \in \sn^0} d^\alpha_{u,v} \s_\alpha$. By comparing the coefficient of $A_w$ in the both side of an equality $\s_v \s_u =\sum_{\alpha \in \sn^0} d^\alpha_{u,v} \s_\alpha$, there exists $\beta \in \sn$ such that $A_\beta$ appears in $\s_v$ and $\alpha=\beta u$ satisfying $\ell(\alpha)=\ell(\beta)+\ell(u)$ and $c^w_u=d^{\alpha}_{u,v}$. Note that existence of such an element $\beta$ is guaranteed if $c^w_u$ is nonzero. Since $\alpha$ is in $\anmz$, so is $u$.\\

Now it is enough to show that $c^{\lambda/d/\mu}_{\nu/e/\emptyset}$ is the same as $c^{\lambda/d-1/\mu}_{\nu/e-1/\emptyset}$ where $\lambda/d-1/\mu, \nu/e-1/\emptyset$ are cylindric shapes of type $(m,n)$ and $e\geq 1$. Let $u$ be $\phi^{-1}(\nu/e/\emptyset)$. Then $c^w_u$ is equal to $c^{\lambda/d/\mu}_{\nu/e/\emptyset}$. Therefore, we can write
$$s_{\lambda/d/\mu}=F_w= \sum_{u\in \anmz} c^w_u F_u = \sum_{\nu/e/\emptyset} c^w_u s_{\nu/e/\emptyset}$$
By Theorem \ref{thm.bij} and Corollary \ref{cor.skew}.\\

The condition $e\geq1$ implies that $s_{i}$ occurs in a reduced word of $u=\phi^{-1} (\nu/e/\emptyset)$ for any $i\in \znz$. By Lemma \ref{lemma.wm2} and Theorem \ref{thm.kschurrm}, $u=u'r_m$ for some $u' \in \anm$. By Lemma \ref{lemma.wm1}, $u$ is in $\anmz$. Therefore, by Proposition \ref{prop.decomp} we have
$$\s_u = \s_{u'} \s_{r_m}.$$
Indeed, if $u=u_0 r_m^{f}$ is a ribbon decomposition, a ribbon decomposition of $u'$ is $u_0 r_m^{f-1}$.
Consider a coefficient of $A_w$ in $\s_{r_m}\s_{u'} $, which is $c^w_u$. By Theorem \ref{thm.kschurrm}, the coefficient of $A_w$ in $\s_{r_m}\s_{u'} $ is $\sum_{\substack{w=ab}} c^b_{u'}$
where $a$ is a $n$-connected ribbon and $w=ab$ with $\ell(w)=\ell(a)+\ell(b)$. By Lemma \ref{lemma.wm2}, there is unique such $a$, and we have $c^w_u=c^b_{u'}$. By construction, we have $b= \phi^{-1}(\lambda/d-1/\emptyset)( \phi^{-1}(\mu/0/\emptyset))^{-1}$ and $u'=\phi^{-1}(\nu/e-1/\emptyset)$, so we showed that
$$c^{\lambda/d/\mu}_{\nu/e/\emptyset}=c^w_u=c^b_{u'}=c^{\lambda/d-1/\mu}_{\nu/e-1/\emptyset}.$$

If $e=0$, let's consider the following:

$$s_{\lambda/d/\mu}(x_1,\ldots,x_m)= \sum c^{\lambda/d/\mu}_{\nu/e/\emptyset} s_{\nu/e/\emptyset}(x_1,\ldots,x_m).$$
Postnikov \cite{postnikov.2005} showed that $s_{\nu/e/\emptyset}(x_1,\ldots,x_m)$ is nonzero if and only if $\nu/e/\emptyset$ is toric. The condition is equivalent to $e=0$. Then $s_{\nu/e/\emptyset}(x_1,\ldots,x_m)$ is equal to $s_\nu(x_1,\ldots,x_m)$, and by (\ref{postnikov}) $c^{\lambda/d/\mu}_{\nu/0/\emptyset}$ is equal to $C^{\lambda,d}_{\mu,\nu}$.
\qed

Note that we also proved the following theorem:
\begin{thm}\label{thm.8}
If $w\in \sn$ is $321$-avoiding, $c^w_u$ is 0 unless $u$ is a $321$-avoiding $0$-Grassmannian element. Moreover, if $w$ is $321$-avoiding and satisfies $\maxc(w) \leq m$ and $ \maxr(w) \leq n-m $, then $c^w_u=0$ unless $u$ is in $\anmz$.
\end{thm}

We can apply Theorem \ref{thm.8} to characterize the basis of $\bm$. Via the projection $\mathbb{B}\rightarrow \bm$, $\bs_u$ is zero when $u \notin \anmz$ since for every term $A_w$ appearing in $\s_u$ we have $w \notin \anm$. On the other hand, $\bs_u$ is not zero when $u$ is in $\anmz$ because $\bs_u$ have the unique $0$-Grassmannian term $A_u$. Therefore, the set $\{\bs_u \mid u\in \anmz \}$ forms a basis of $\bm$.

\section{Effective algorithm for the computation}
In this section, we provide an effective algorithm to compute the expansion of the cylindric skew Schur function in terms of cylindric Schur functions, and the expansion of the affine Stanley symmetric function in terms of affine Schur functions. The strategy is to use the following dual Pieri rule of the affine Stanley symmetric functions.

\begin{thm} \cite{LLMS} \label{thm.dualpieri}
For any $m\leq n$, we have
$$h_m^\perp(F_w)=\sum_{w=uv} F_v =\sum_{w=vu} F_v$$
where the summations run over all cyclically decreasing elements $u$ of length $m$ satisfying $\ell(w)=\ell(u)+\ell(v)$.
\end{thm}

To describe the algorithm, we define a total order on a set of partitions. For two partitions $\mu,\nu$, we denote $\mu<\nu$ when either $|\mu|<|\nu|$, or $|\mu|=|\nu|$ and $\mu$ is less than equal to $\nu$ with respect to the reverse colexicographic ordering, i.e., there exists $a>0$ such that $\mu_a>\nu_a$ and $\mu_b=\nu_b$ for all $b>a$. Here, for a partition $\lambda$ we set $\lambda_b=0$ if $b>\ell(\lambda)$.\\

For $w\in \sn$ and $v \in \sn$ such that $wv$ is $p$-Grasssmannian for some $p\in \znz$, consider a maximal cyclically decreasing decomposition $d_{J_\ell}\cdots d_{J_1}$ of $v$ and we denoted the partition $(|J_1|,\ldots,|J_\ell|)$ by $\lambda(v)$. We will show that $F_w$ can be written as a summation $\sum_u \pm F_u$ such that for any $u\in \sn$ there exists $v_u \in \sn$ such that $u v_u$ is $p$-Grassmannian and $\lambda(v_u)<\lambda(v)$. Then by repeating the process on each $F_u$ when $\lambda(v_u)$ is not the empty partition, $F_w$ is equal to a linear combination of $F_u$ with $\lambda(u)=\emptyset$, which means that all $u$ is $p$-Grassmannian and we are done. The algorithm preserves the $321$-avoiding condition, namely if $w$ is in $\anm$ for some $m$, then $u\in \anm$ for all $u$ appearing in the summation.\\

First step is to show that for each $w\in \sn$, we find a small $v\in \sn$ such that $wv$ is $p$-Grassmannian for some $p\in \znz$.
  
\begin{lemma} \label{lemma.upperbound1}
For any affine permutation $w$, there exists $v$ such that $\ell(wv)=\ell(w)+\ell(v)$, $\ell(v)\leq (k-1)\cdot 1 +(k-2)\cdot 2+ \cdots +1 \cdot (k-1)$, and $wv$ is $p$-Grassmannian for some $p\in \znz$.
\end{lemma}

\begin{remark}
When we omit the condition $\ell(v)\leq (k-1)\cdot 1 +(k-2)\cdot 2+ \cdots +1 \cdot (k-1)$, the theorem follows from \cite{morse.schilling.2016}. However, in their proof $\ell(v)$ is always bigger than or equal to $(k-1)\cdot 1 +(k-2)\cdot 2+ \cdots +1 \cdot (k-1)$.
\end{remark}When $w$ is in $\anm$ for some $m$, the upper bound on $\ell(v)$ is much smaller.
\begin{lemma} \label{lemma.upperbound2}
If $w$ is in $\anm$,  there exists $v \in \anmz$ such that $\ell(wv)=\ell(w)+\ell(v)$, $\ell(v)\leq (n-m)(m-1)/2$, and $wv$ is $i$-Grassmannian for some $p \in \znz$.
\end{lemma}

We postpone a proof of two lemmas to the end of the section and describe the algorithm assuming the lemmas first.\\ 

For an affine permutation $w$, assume that we have $v\in \sn$ such that $wv$ is $p$-Grassmannian for some $p$. First of all, since $F_w$ is equal to $F_{f(w)}$ where $f:\sn \rightarrow\sn$ is the map sending $s_i$ to $s_{i+1}$, we can assume that $p$ is equal to zero. Let $\lambda=(\lambda_1,\ldots,\lambda_\ell)$ be $\lambda(v)$. Since $v$ is $0$-Grassmannian, $v$ is uniquely determined by $\lambda$. Let $v'$ be the $0$-Grassmannian element corresponding to $(\lambda_1,\ldots,\lambda_{\ell-1})$, so that we have
$$v= d_{[-\ell+1,\lambda_\ell-\ell]}v'$$

Now we apply Theorem \ref{thm.dualpieri} to $w'=wd_{[-\ell+1,\lambda_\ell-\ell]}$ and $m=\lambda_\ell$. Recall that $B_-(w)$ is the set of $w'u_J$ for a subset $J\subset \znz$ with $|J|=\lambda_\ell$, $\ell(wu_J)=\ell(w)-|J|$ and $J\neq [-\ell+1,\lambda_\ell-\ell]$ and $B_+(w)$ is the set of $u_Jw$ for a subset $J\subset \znz$ with $|J|=\lambda_\ell$ and $\ell(u_J w)=\ell(w)-|J|$. Then we have
$$F_w=\sum_{u \in B_+(w)}F_u - \sum_{u\in B_-(w)} F_u.$$
Here we used the identity $(d_J)^{-1}=u_J$.\\

Note that for $u=u_Jw \in B_+(w)$ $uv'$ is $0$-Grassmannian, since $wv'=d_Juv'$ is $0$-Grassmannian and $\ell(d_Juv')=|J|+\ell(u)+\ell(v')$. Clearly, we have $\lambda(v)>\lambda(v')$. For $u=w u_J \in B_-(w)$, we know that $wv'=ud_Jv'$ and $d_Jv'$ are $0$-Grassmannian. For $J\subset \znz$ with $|J|=\lambda_\ell$, we have $\lambda(d_Jv')\leq \lambda(v)$ by the definition of $v'$ and the equality holds for the unique $J$ which is $[-\ell+1,\lambda_\ell-\ell]$ and in this case $d_Jv'=v$. Therefore, we have $\lambda(d_Jv')<\lambda(v)$ for other $J$, and we obtained an algorithm to compute the expansion of the affine Stanley symmetric function in terms of affine Schur functions by repeating the procedure.\\

When $w$ is in $\anm$ for some $m$, we can find $v\in \anm$ such that $wv$ is $p$-Grassmannian by Lemma \ref{lemma.upperbound2}. Since every element in $B_-(w)$ and $B_+(w)$ is of the form $u_Jw$ or $wu_J$ with the length equal to $\ell(w)-|J|$ these elements are also in $\anm$, and by repeating the procedure we get the expansion of $F_w$ in terms of $F_u$ where $u \in \anm$ is $0$-Grassmannian. By Theorem \ref{thm.bij}, $F_u$ is equal to the cylindric Schur function $s_{\phi(w)}$ and we are done. See Example 2 for an example of the algorithm.\\

Now we prove Lemma \ref{lemma.upperbound1} and Lemma \ref{lemma.upperbound2}.

{\it Proof of Lemma \ref{lemma.upperbound1}.}
For $w\in \sn$ and $i\in \mathbb{Z}$, let $c_i(w)$ be the number of $j<i$ such that $w(j)>w(i)$. Note that $w$ is $p$-Grassmannian if and only if $w(p+1)<w(p+2)<\cdots<w(p+n)$. This condition is equivalent to $c_{p+1}(w)\geq c_{p+2}(w)\geq \cdots \geq c_{p+n}(w)$, and in this case $c_{p+n}(w)=0$ and a partition $(c_{p+1}(w),c_{p+2}(w),\ldots,c_{p+n}(w))$ is the same as the transpose of $\lambda(w)$. When $w$ is a general element in $\sn$, we always have $c_{j}=c_{j+n}$ for any $j\in \mathbb{Z}$, and one can show that there exists $j$ such that $c_j=0$. See \cite{denton.2012} for example.\\

For given $w\in \sn$, we inductively construct $v_i\in \sn$ for $i=1,\ldots,k-2$ such that $w_i=v_1v_2 \cdots v_i$ satisfy $\ell(ww_i)=\ell(w)+\sum_{j=1}^i \ell(v_j)$, $c_{q_i+1}(ww_i)\geq c_{q_i+2}(ww_i)\geq \cdots \geq c_{q_i+i+1}(ww_i)$ for some $q_i \in \mathbb{Z}$, and $c_{q_i+j}(ww_i)$ is less than equal to $c_{q_i+i+1}(ww_i)$ for $q_i+i+1<j\leq q_i+n$. We also require that $\ell(v_i)\leq i(k-i)$, so that $ww_{k-1}$ is $q_{k-1}$-Grassmannian and $\ell(w_{k-1}) \leq \sum_{i=1}^{k-1} i(k-i)$, which proves the lemma.\\

For $w\in \sn$, choose an integer $i_1$ with maximal $c_{i_1}(w)$, and choose $i_1<j_1<i_1+n$ such that $c_{j_1}(w)$ is the largest number in a set $\{c_p(w)\mid i_1<p<i_1+n\}$. If there are multiple choices, we just choose one of them. Define $v_1$ to be $s_{i_1}s_{i_1+1}\ldots s_{j_1-2}=d_{[i_1,j_1-2]}$ if $j_1>i_1+1$, and the identity if $j_1=i_1+1$.
One can show that $c_p(wv_1)=c_p(w)$ if $p$ is not in $[i_1,j_1-1]$ modulo $n$, $c_p(wv_1)=c_{p-1}(w)$ for $p$ in $[i_1,j_1-2]$, and $c_{j_1-1}(wv_1)=c_{i_1}(w)+j_1-i_1-1$. Since $c_{i_1}$ is the largest number among $c_j$'s, $\ell(wv_1)=\ell(w)+\ell(v_1)$. Therefore, when we take $q_1=j_1-2$ so that $c_{q_1+1}(wv_1)>c_{q_1+2}(wv_2)$ and other $c_j$'s are less than equal to $c_{q_1+2}(wv_2)$. Note that $\ell(v_1)$ is at most $k-1$.\\

We will construct $v_{i+1}$ similarly for $i<k-1$, assuming that we are given $v_1,\ldots,v_{i-1}$. Note that there exists $q_i \in \mathbb{Z}$ such that $c_{q_i+1}(ww_i)\geq c_{q_i+2}(ww_i)\geq \cdots \geq c_{q_i+i+1}(ww_i)$. Let $c_{j}(w w_i)$ be the largest number in a set $\{c_p(ww_i)\mid q_i+i+1<p\leq q_i+n\}$. If $j$ is $q_i+i+2$, then we can take $v_{i+1}$ to be the identity and we are done. If $j>q_i+i+2$, then we take
$$v_{i+1}= d_{[q_i+i+1,j-2]}d_{[q_i+i,j-3]}\cdots d_{[q_i+1,j-i-2]}.$$
Since $c_j(w w_i)$ is less than or equal to $c_{q_i+1}(ww_i),\ldots,c_{q_i+i+1}(ww_i)$, $\ell(ww_iv_{i+1})=\ell(ww_i)+\ell(v_{i+1})$. One can check that $c_p(ww_{i+1})=c_p(ww_i)$ if $p$ is not in $[q_i+1,j-1]$ modulo $n$, $c_p(wv_1)=c_{p-i-1}(w)$ for $p$ in $[q_i+1,j-i-2]$, and $c_{p}(wv_1)=c_{p-(j-q_i-i-2)}(w)+(j-q_i-i-2)$ for $p$ in $[j-i-1,j-1]$.
Also note that $\ell(v_{i+1})\leq (i+1)(j-q_i-i-2)\leq(i+1)(k-i-1)$. Therefore, one can take $q_{i+1}=j-i-1$ and we are done.
\qed \\

{\it Proof of Lemma \ref{lemma.upperbound2}.}
Although the proof of the lemma can be proved similarly, it is much better to visualize using cylindric shapes. Lam \cite{lam.2006} showed that for a $321$-avoiding affine permutation $w$, there exists $q<n$ such that $F_w$ is equal to $F_{\lambda[r]/\mu[s]}$ where $\lambda[r]/\mu[s]$ is a cylindric shape of type $(q,n)$ for some $q$. If $s_i$ appears at least once in $w \in \anm$ for all $i$, such $q$ must be unique which is the same as $m$. Indeed, by Lemma \ref{lemma.wm2} and Lemma \ref{lemma.ribbon2} $w$ can be written as $uv$ where $\ell(w)=\ell(u)+\ell(v)$ and $v$ is a $n$-connected ribbon $u_{J^c}d_J$. Then the $n$-connected ribbon determines $q$ by $|J^c|=m$.\\

When $s_i$ does not appear in $w\in \sn$ for some $i\in \znz$, one can still find a cylindric shape of type $(m,n)$, but we don't really need to consider this since $F_w$ is equal to single skew Schur function \cite{billey.stanley.1993} and the coefficients in the Schur expansion of a skew Schur function are just Littlewood-Richardson coefficients.  \\

To prove the lemma, it is enough to find a bi-infinite increasing $(m,n)$-periodic sequence $\alpha$ such that $D_\alpha\subset D_\mu[s]$ and $\alpha_{i+1}=\alpha_{i+2}=\cdots=\alpha_{i+m}$ for some $i$, and $|D_{\mu[s]/\alpha}|\leq (n-m)(m-1)/2$. If this happens, for $w\in \anm$ satisfying $A_w\cdot(\alpha/\alpha)=(\mu[s]/\alpha)$ then $w$ is $(\alpha_{i+m}-i-m)$-Grassmannian, which proves the lemma. There are many ways to define $\alpha$ satisfying the above-mentioned conditions. Let $\beta$ be $\mu[s]$. For any $a=1,\ldots,m$, we can define a $(m,n)$-periodic bi-infinite sequence $\alpha(a)$ to be $\alpha(a)_i=\beta_a$ for $a\leq i <a+m$. It is clear that $\alpha(a)_i= \alpha(a)_a=\beta_a<\beta_i$ for $a\leq i <a+m$ so that $\beta/\alpha$ is a cylindric shape. Moreover, $|D_{\beta/\alpha(a)}|$ is equal to 
$$\sum_{i=a}^{a+m-1} (\beta_i-\beta_a)=\sum_{i=1}^m \beta_i  - m\cdot \beta_a + (n-m)(a-1).$$
Therefore, the average of $|D_{\beta/\alpha(a)}|$ for $a=1,\ldots,m$ is 
$${\sum_{a=1}^m (n-m)(a-1)\over m}= {(n-m)(m-1)\over 2},$$
so there exists $a$ such that $$|D_{\beta/\alpha(a)}|\leq {(n-m)(m-1)\over 2}$$
and we are done.\qed 

\begin{example}
Let $n=6,m=3$ and $w=s_5s_3s_1s_4s_2s_0$. We write $w=531420$ for the sake of simplicity. One can take $v$ to be $510$ so that $wv$ is $0$-Grassmannian. Let $\lambda/r/\mu$ be $\phi(wv)$. Then one can easily check that $\lambda=(2,1), r=1$ and $\mu$ is empty partition. Since $v$ corresponds to a partition $(2,1)$ via the bijection between $P_{mn}$ and a partition contained in $(n-m)^m$, $w$ corresponds to a cylindric shape $(2,1)/1/(2,1)$. \\

In this setup, we apply Theorem \ref{thm.dualpieri} to $w'=ws_5$ and $m=1$ to obtain
$$F_w=\sum_{u \in B_+(w)}F_u - \sum_{u\in B_-(w)} F_u.$$
where $B_+(w)=\{541052,341052,354052\}$ and $B_-(w)=\{354105\}$. See Figure 1 for cylindric shapes corresponding to $w,B_{+}(w)$ and $B_{-}(w)$.\\

Note that except for $w$ and $341052$, other elements are missing $s_i$ for some $i$. For example, $541052$ does not use $s_3$ in a reduced word. Therefore, $F_{541052}$ is equal to skew Schur function $s_{(3,3,2)/(2)}$, which is immediate from Figure 1 and Corollary \ref{cor.skew}. Similarly, we have $F_{354052}=s_{(3,3,2)/(1,1)}$ and $F_{354105}=
s_{(3,3,2,1)/(3)}=s_{(3,2,1)}$. Although in general we don't this step in the algorithm, it clearly reduces computational complexity.

To sum up, we have
\begin{align*}
F_w&=F_{541052}+F_{341052}+F_{354052}-F_{354105}\\
&=F_{341052}+s_{(3,3,2)/(1,1)}+s_{(3,3,2)/(2)}-s_{(3,2,1)}.
\end{align*}
Now we can repeat the procedure to $F_{341052}$. We know that $(341052)(10)$ is $0$-Grassmannian and $\lambda(s_1s_0)<\lambda(s_5s_1s_0)=\lambda(v)$ where the order $<$ is defined in Section 5. By applying Theorem \ref{thm.dualpieri} to $34105210$ and $m=2$, we obtain
$$F_{341052}=F_{345210}+F_{405210}$$
and both $345210$ and $405210$ are $0$-Grassmannian. Therefore, we have

\begin{align*}
F_w&=F_{345210}+F_{405210}+s_{(3,3,2)/(1,1)}+s_{(3,3,2)/(2)}-s_{(3,2,1)}\\
&=F_{345210}+F_{405210}+s_{(2,2,2)}+s_{(3,3)}+s_{(3,2,1)}\\
&=F_{345210}+F_{405210}+F_{540510}+F_{105210}+F_{405210}\\
&=F_{345210}+2F_{405210}+F_{540510}+F_{105210}.
\end{align*}

\begin{figure}
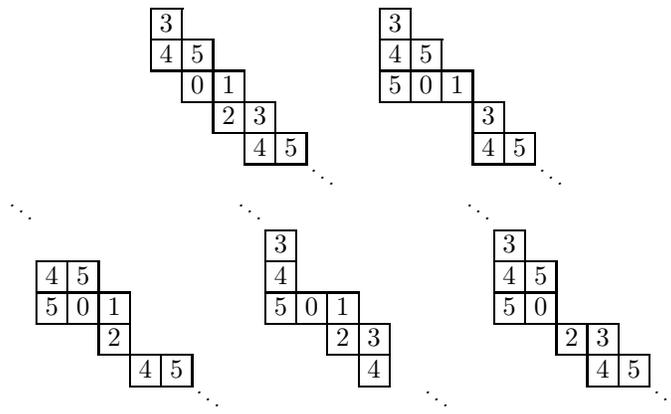

\caption{Cylindric shapes for $w,354105,541052,341052, 354052$.}
\begin{center}
\ytableausetup{mathmode,boxsize=4mm}
\begin{ytableau}
\none[\scalebox{0.8}{$\ddots$}]\\
\none&3\\
\none&4&5\\
\none&\none&0&1\\
\none&\none&\none&2&3\\
\none&\none&\none&\none&4&5\\
\none&\none&\none&\none&\none&\none&\none[\scalebox{0.8}{$\ddots$}]
\end{ytableau}
 \begin{ytableau}
\none[\scalebox{0.8}{$\ddots$}]\\
\none&3\\
\none&4&5\\
\none&5&0&1\\
\none&\none&\none&\none&3\\
\none&\none&\none&\none&4&5\\
\none&\none&\none&\none&\none&\none&\none[\scalebox{0.8}{$\ddots$}]
\end{ytableau}

 \begin{ytableau}
 \none[\scalebox{0.8}{$\ddots$}]\\
\none\\
\none&4&5\\
\none&5&0&1\\
\none&\none&\none&2\\
\none&\none&\none&\none&4&5\\
\none&\none&\none&\none&\none&\none&\none[\scalebox{0.8}{$\ddots$}]
\end{ytableau}
\begin{ytableau}
\none[\scalebox{0.8}{$\ddots$}]\\
\none&3\\
\none&4\\
\none&5&0&1\\
\none&\none&\none&2&3\\
\none&\none&\none&\none&4\\
\none&\none&\none&\none&\none&\none&\none[\scalebox{0.8}{$\ddots$}]
\end{ytableau}
\begin{ytableau}
\none[\scalebox{0.8}{$\ddots$}]\\
\none&3\\
\none&4&5\\
\none&5&0\\
\none&\none&\none&2&3\\
\none&\none&\none&\none&4&5\\
\none&\none&\none&\none&\none&\none&\none[\scalebox{0.8}{$\ddots$}]
\end{ytableau}\end{center}
\end{figure}
\end{example}

\section{Concluding remark}

It would be amazing to find a manifestly positive formula for the expansion of the cylindric skew Schur functions in terms of cylindric Schur functions, or the expansion of the affine Stanley symmetric functions in terms of affine Schur functions. Although the latter problem would be hard at this time since it contains famous problems such as 3-point Gromov-Witten invariants of the flag variety, the formal problem clearly have more combinatorics so that one may hope that some of classical combinatorics on skew shapes, such as jeu de taquin, can be generalized. However, naive generalization of jeu de taquin has a problem. See Yoo's thesis \cite{Yoo.2011}. It would be interesting to develop combinatorics related to cylindric shapes to attack the problem. \\


It would be also interesting to define equivariant version of cylindric skew Schur functions. Although double affine Stanley symmetric functions are defined by Lam and Shimozono \cite{LS.2013}, their definition does not seem to contain much combinatorics. For example, there is no known tableau definition of double affine Stanley symmetric function when $w$ is 321-avoiding, whereas there is a tableau definition for the double skew Schur functions \cite{Molev.2009}.

\bibliographystyle{alpha}
\bibliography{paper}
\end{document}